\documentclass[11pt]{article}

\usepackage[applemac]{inputenc}
\usepackage{amsmath}
\usepackage{amssymb}
\usepackage{amsthm}
\usepackage{textcomp}
\usepackage{graphicx}

\newtheorem{defn}{Definition}[section]
\newtheorem{lemma}[defn]{Lemma}
\newtheorem{prop}[defn]{Proposition}
\newtheorem{theo}[defn]{Theorem}
\newtheorem{coro}[defn]{Corollary}

\newtheorem{claim}{Claim}

\def\Ric{\mathop{\rm Ric}\nolimits}
\def\Hess {\mathop{\rm Hess}\nolimits}
\def\Rm{\mathop{\rm Rm}\nolimits}
\def\det{\mathop{\rm det}\nolimits}
\def\trace{\mathop{\rm trace}\nolimits}
\def\diam{\mathop{\rm diam}\nolimits}
\def\vol{\mathop{\rm vol}\nolimits}
\def\eucl{\mathop{\rm eucl}\nolimits}
\def\dim{\mathop{\rm dim}\nolimits}
\def\vol{\mathop{\rm Vol}\nolimits}
\def\spec{\mathop{\rm Spec}\nolimits}
\def\inj{\mathop{\rm inj}\nolimits}
\def\Isom{\mathop{\rm Isom}\nolimits}

\begin{document}
\author{Alix Deruelle}
\title{Steady gradient Ricci soliton with curvature in $L^1$}
\date{\today}
\maketitle

\begin{abstract}
We characterize complete nonnegatively curved steady gradient soliton with curvature in $L^1$.
We show that there are isometric to a product $((\mathbb{R}^2,g_{cigar})\times (\mathbb{R}^{n-2}, \eucl))/\Gamma $ where $\Gamma$ is a Bieberbach group of rank $n-2$. We prove also a similar local splitting result under weaker curvature assumptions.
\end{abstract}

 \section{Introduction}\label{Intro}

A \textbf{steady  gradient Ricci soliton} is a triple $(M^n,g,\nabla f)$ where $(M^n,g)$ is a Riemannian manifold and $f$ is a smooth function on $M^n$  such that $\Ric=\Hess (f).$
It is said \textbf{complete} if the vector field $\nabla f$ is complete.

In this paper, we prove a rigidity result for steady gradient soliton of nonnegative sectional curvature with curvature in $L^1(M^n,g)$.

\begin{theo}\label{global-iso}
Let $(M^n,g,\nabla f)$ be a complete nonflat steady gradient Ricci soliton such that

\begin{itemize}
\item[(i)] $K\geq 0$, where $K$ is the sectional curvature of $g$, 
\item[(ii)] $R\in L^1(M^n,g)$.
\end{itemize}

Then  any soul of  $M^n$ has codimension $2$ and is flat. Moreover, the universal covering of $M^n$ is isometric to
$$(\mathbb{R}^2,g_{cigar})\times (\mathbb{R}^{n-2}, \eucl), $$
and $\pi_1(M^n)$ is a Bieberbach group of rank $n-2$.

\end{theo}

This result is relevant in dimensions greater than two. Indeed, the cigar soliton is the only two-dimensional nonflat steady gradient soliton, see [\ref{Chow}] for a proof.
Moreover, condition (i) is always true for an ancient solution of dimension $3$ with bounded curvature on compact time intervals because of the Hamilton-Ivey estimate ([\ref{Chow}], Chap.6, Section 5 for a detailed proof). Hence the following corollary.

\begin{coro}\label{global-iso-3}
Let $(M^3,g,\nabla f)$ be a complete nonflat steady gradient Ricci soliton such that
$$R\in L^1(M^3,g).$$
Then $(M^3,g)$ is isometric to
$$((\mathbb{R}^2,g_{cigar})\times \mathbb{R})/\langle(t,\theta,u)\rightarrow (t,\theta+\alpha,u+a)\rangle,$$
with $(\alpha,a)\in\mathbb{R}\times \mathbb{R}^*$.
\end{coro}

We make some remarks about Theorem \ref{global-iso}.
First we recall the definition of the \textbf{cigar soliton} discovered by Hamilton.
 The cigar metric on  $\mathbb{R}^2$, in special radial coordinates, is $g_{cigar}:=ds^2+\tanh^2 sd\theta^2.$
A standard  calculation shows that $R(g_{cigar})=16/(e^s+e^{-s})^{-2}.$
The curvature is positive and decreases exponentially, moreover this metric is asymptotic to a cylinder of radius $1$, therefore $R\in L^1(\mathbb{R}^2, g_{cigar}).$

Theorem \ref{global-iso} can be seen as a gap theorem in the terminology of Greene-Wu for the curvature decay of steady gradient solitons. In fact, one could assume in Theorem \ref{global-iso} that the scalar curvature decays faster than $1/r^{1+\epsilon}$ for $\epsilon>0$ instead of $R\in L^1(M^n,g)$. The proof is quite the same. Then the result is that the curvature decays exponentially. In this way, let us mention a result of Greene and Wu [\ref{Greene-Wu}] on nonnegatively curved spaces which are flat at infinity.

\begin{theo}[Greene-Wu]\label{flat-infinity}
Let $M^n$ be a complete noncompact Riemannian manifold of nonnegative sectional curvature. If $M^n$ is flat at infinity, then either $(a)$ $M^n$ is flat or $(b)$ any soul of $M^n$ is flat and has codimension $2$, the universal covering of $M^n$ splits isometrically as $\mathbb{R}^{n-2}\times \Sigma_0$ where $\Sigma_0$ is diffeomorphic to $\mathbb{R}^2$ and is flat at infinity but not flat everywhere, and finally the fundamental group of $M^n$ is a Bieberbach group of rank $n-2$.
\end{theo} 

The idea of the proof of Theorem \ref{global-iso} consists in analysing the volume and the diameter of the level sets of $f$ where $f$ is seen like a Morse function. By a blow-up argument inspired by Perelman's proof of classification of $3$-dimensional shrinking gradient Ricci soliton [\ref{Perelman}], we show that the level sets of $f$ are diffeomorphic to flat compact manifolds. In fact, by Bochner's theorem, the level sets of $f$ are metrically flat. From this, the global splitting of the universal cover is easily obtained. Finally, we follow closely the arguments of the proof of Theorem \ref{flat-infinity} to show that any soul has codimension $2$.\\

\textbf{Organization.} In section \ref{section1}, we recall basic equations for steady solitons, and study the topological structure at infinity. We establish some volume and diameter estimates of the level sets of $f$. In section \ref{section2}, we prove Theorem \ref{global-iso}, and, under weaker curvature assumptions, a local splitting result (cf. Theorem \ref{local-iso}). We give also a rigidity result on steady breathers. In section \ref{section3}, we make some further remarks about the link between the volume growth and the scalar curvature decay of a steady soliton, 
and give necessary geometric conditions to have a positive Perelman's functional  on steady gradient solitons.\\

\textbf{Acknowledgements.}
I would like to thank my advisor Laurent Bessières for his encouragement and his precious enlightenment on this problem.

 \section{Geometry and topology of the level sets of $f$}\label{section1}
 
 We begin by recalling the link   ([\ref{Chow}], Chap.4)  between steady gradient solitons and the Ricci flow.

\begin{theo}
If $(M^n,g,\nabla f)$ is a complete steady gradient Ricci soliton then there exist a solution $g(t)$ to the Ricci flow with $g(0)=g$, a family of diffeomorphisms $(\psi_t)_{t\in\mathbb{R}}$ with $\psi_0=Id_{M^n}$ and functions $f(t)$ with $f(0)=f$ defined for all $t\in \mathbb{R}$ such that

\begin{itemize}
\item[(i)] $\psi_t:M^n\rightarrow M^n$ is the 1-parameter group of diffeomorphisms generated by $-\nabla^g f$,
\item[(ii)] $g(t)={\psi_t}^*g$, i.e. for all $t\in \mathbb{R}$, i.e. $g(t)$ is isometric to $g$,
\item[(iii)] $f(t)={\psi_t}^*f$, for all $t\in \mathbb{R}$.
\end{itemize}
\end{theo}

Therefore, a steady soliton is an ancient solution to the Ricci flow, i.e. defined  on an interval $]-\infty,\omega)$, where $\omega$ can be $+\infty$. Let us quote from ([\ref{Chow}],Chap. 2; Lemma 2.18) a result which follows from the strong and weak maximum principles for heat-type equations:

\begin{lemma}[Nontrivial ancient solutions have positive scalar curvature ]
If $(M^n,g(t))$ is a complete ancient solution (i.e. $(M^n,g(t))$ is complete for all $t \in (-\infty,\omega)$) to the Ricci flow with bounded curvature on compact time intervals then, either $R(g(t))>0$  for all $t\in (-\infty,\omega)$, either  $\Ric(g(t))=0$  for all $t\in (-\infty,\omega)$.
\end{lemma}

Next, we collect the basic identities satisfied by a steady gradient soliton  ([\ref{Chow}], Chap.4). 

\begin{lemma}\label{id}
Let $(M^n,g,\nabla f)$ be a complete steady gradient soliton. Then:

\begin{itemize}
\item[(i)] $\Delta f = R$, where $\Delta :=\trace_g\Hess$,
\item[(ii)]$\nabla R+ 2\Ric(\nabla f)=0$,
\item[(iii)]$\arrowvert \nabla f \arrowvert^2+R=Cst.$
\end{itemize}
\end{lemma}
 
 By the third identity, a  steady  soliton  with bounded curvature is always complete because  $\arrowvert\nabla f\arrowvert$ is bounded. In the sequel, we only consider steady gradient Ricci solitons with positive scalar curvature and bounded curvature. Such solitons are necessarily noncompact.

\begin{lemma}[Topological structure at infinity]\label{structure}
Let $(M^n,g,\nabla f)$ be a steady gradient soliton such that $\Ric\geq 0$ with $R>0$.
Suppose that,
$$\lim_{+\infty}R=0.$$
Then $R$ attains its supremum, $R_{max}$, at a point $p$, and on $M^n$,
$$\arrowvert \nabla f \arrowvert^2+R=R_{max}.$$ 
Moreover, $f$ attains its minimum at $p$ and there exist constants $c_i=c_i(M^n, f,R)$ ($i=1...6$) such that
\begin{itemize}
\item[(i)] $\arrowvert\nabla f\arrowvert\leq c_1$, on $M^n$,
\item[(ii)] $c_2\leq\arrowvert\nabla f\arrowvert$, at infinity,
\item[(iii)]$c_3r_p(x)+c_4\leq f(x)\leq c_5r_p(x)+c_6$ at infinity, where $r_p$ is the distance function centered at $p$.
\end{itemize}
In particular,  $M^n$ has finite topological type.
\end{lemma}

Note that José A. Carrillo and Lei Ni [\ref{Carrillo}] have shown a similar lemma under weaker hypotheses: $\limsup_{x\rightarrow+\infty}R<\sup_{M^n}R=R(p)=R_{max}$ ($R$ is supposed to attain its supremum). To be complete, we give a short proof of this lemma following [\ref{Carrillo}].\\

\textbf{Proof of Lemma \ref{structure}}.
As the scalar curvature $R$ tends to $0$ at infinity, it attains its maximum $R_{max}>0$ at a point $p$ of $M^n$. Moreover, we know that there exists a constant $C>0$ such that $\arrowvert \nabla f \arrowvert^2+R=C.$ In particular, $C\geq R_{max}$.
Assume that $R_{max}<C$. Consider the flow $(\psi_t(p))_t$ generated by the vector field $\nabla f$. This flow is defined on $\mathbb{R}$ because $\nabla f$ is complete. define the function $F(t):=f(\psi_t(p))$ for $t\in \mathbb{R}$. Then,
$$F'(t)=\arrowvert \nabla f\arrowvert^2\quad\mbox{and}\quad F''(t)=2\Ric(\nabla f,\nabla f),$$
implicitely evaluated at the point $\psi_t(p)$.
By assumption, $F'(t)\geq C-R_{max}>0$ and $F'(0)=C-R_{max}=\min_{M^n}\arrowvert \nabla f\arrowvert^2$.
 Now $F''(t)\geq 0$ for $\Ric\geq 0$, i.e. $F'$ is a non-decreasing function on $\mathbb{R}$. So $F'$ is constant on $]-\infty,0]$ and $F(t)=(C-R_{max})t +f(p)$ for $t\leq 0$.
 In particular, $\lim_{t\rightarrow-\infty}F(t)=-\infty$. As $f$ is continue (since it is smooth!) on $M^n$, this implies that $(\psi_t(p))_t$ is not bounded for $t\leq 0$.
Therefore, there exists a subsequence $t_k\rightarrow -\infty$ such that $r_p(\psi_{t_k}(p))\rightarrow+\infty$.
Thus, $\lim_{k\rightarrow+\infty}R(\psi_{t_k}(p))=0$.
Now, $F''(t)=2\Ric(\nabla f,\nabla f)=-g(\nabla R,\nabla f)=0$ for $t\leq 0$, i.e. $R(\psi_t(p))=R_{max}>0$ for $t\leq 0$. Contradiction.
We have shown that $\arrowvert \nabla f \arrowvert^2+R=R_{max}.$
Then $\nabla f(p)=0$ and as $\Hess(f)(p)\geq 0$, $f$ attains its minimum at $p$.
We deduce that $\liminf_{x\rightarrow+\infty}\arrowvert \nabla f \arrowvert^2\geq \lambda>0$.
Finally, we can show the inequalities satisfied by $f$ as in [\ref{Carrillo}].

 \hfill$\square$\\

Remember that the critical set of a convex function is exactly the set where it attains its minimum. With the notations of the previous lemma,  we have 
$$ \{f=\min_{M^n} f\}=\{\nabla f=0\}=\{R=R_{max}\}.$$

In the following, we suppose that $\min_{M^n} f=0$.
Now, consider the compact hypersurfaces $M_t:=f^{-1}(t)$, levels of $f$, for $t$ positive. We will also denote the sublevels (resp. superlevels) of $f$ by $M_{\leq t}:=f^{-1}(]-\infty,t])$ (resp. $M_{\geq t}:=f^{-1}([t,+\infty[)$).

Let $(\phi_t)_t$ be the $1$-parameter group of diffeomorphisms generated by the vector field $\nabla f/\arrowvert\nabla f\arrowvert^2$ defined on $M^n\setminus M_0$. For $t_0>0$, $\phi_{t-t_0}$ is a diffeomorphism between $M_{t_0}$ and $M_t$ for $t\geq t_0$. Outside a compact set,  $M^n$ is diffeomorphic to $[t_0, +\infty[\times M_{t_0}$ for $t_0>0$. We suppose $n>2$.

\begin{prop}[Volume estimate]\label{vol-hyper}
Let $(M^n,g,\nabla f)$ be a complete steady gradient  soliton such that
\begin{itemize}
\item[(i)] $\Ric\geq 0$ and $R>0$,
\item[(ii)] $\lim_{+\infty}R=0.$
\end{itemize}

Then
$$0\leq A'(t)\leq c(t_0)\int_{M_t}\frac{R}{\arrowvert\nabla f\arrowvert}dA_t, \quad \forall t\geq t_0, $$ where $A(t):=Vol_{g_t}M_t$,
$g_t$ is the induced metric on $M_t$ by $g$ and $c(t_0)$ is a positive constant depending on $\inf_{M_{\geq t_0}}\arrowvert \nabla f\arrowvert$.

\end{prop}

\textbf{Proof of Proposition \ref{vol-hyper}}.
The curvature assumptions allows to apply Lemma \ref{structure}. Thus, the hypersurfaces $M_t$ are well-defined for $t>0$.
The flow of the hypersurface $M_t$ satisfies $\frac{\partial\phi_t}{\partial t} = (\nabla f/\arrowvert\nabla f\arrowvert^2)(\phi_{t}).$ Therefore, the first variation formula for the area of $M_t$ is given by $$A'(t)=\int_{M_t}\frac{H_t}{\arrowvert \nabla f\arrowvert} dA,$$
where $H_t$ is the mean curvature of $M_t$.
Now the second fundamental form of $M_t$ is 
$$h_t:=\frac{\Hess (f)}{\arrowvert \nabla f\arrowvert}=\frac{\Ric}{\arrowvert \nabla f\arrowvert}.$$
So, 
$$A'(t)=\int_{M_t}\frac{R-\Ric(\textbf{n},\textbf{n})}{\arrowvert \nabla f\arrowvert^2}dA,$$
where $\textbf{n}:=\nabla f/\arrowvert \nabla f\arrowvert$ is the unit outward normal to the hypersurface $M_t$.
The first inequality comes from the nonnegativity of the Ricci curvature. 
The second one is due to the nonnegativity of the Ricci curvature and to the uniform boundedness from below of $\arrowvert \nabla f\arrowvert$ on $M_{\geq t_0}:=\{f\geq f(t_0)\}$.
\hfill$\square$\\

We deduce the following corollary by the co-area formula.

\begin{coro}\label{vol-borne}
Let  $(M^n,g,\nabla f)$ a complete steady gradient soliton satisfying the hypotheses of Proposition \ref{vol-hyper}. Then,

$$A(t_0)\leq A(t)\leq A(t_0)+c(t_0)\int_{M_{t_0\leq s\leq t}}R d\mu, \quad \forall t\geq t_0.$$

\end{coro}

Consequently, a steady gradient soliton satisfying the assumptions of Proposition \ref{vol-hyper} with $R\in L^1(M^n,g)$ has linear volume growth, i.e., for any $p\in M^n$, there exist positive constants $C_1$ and $C_2$ such that for all $r$ large enough,

$$C_1r\leq \vol B(p,r)\leq C_2r.$$

\begin{prop}[Comparison of the metrics $\phi_{t-t_0}^*g_t$ and $g_{t_0}$]\label{long}

Let $(M^n,g,\nabla f)$ be a complete steady gradient soliton satisfying the assumptions of Proposition \ref{vol-hyper}. Let $V$ be a vector field tangent to $M_{t_0}$.
Then,
$$g_{t_0}(V,V)\leq (\phi_{t-t_0}^*g_t)(V,V), \quad \forall t\geq t_0.$$

\end{prop}

\textbf{Proof of Proposition \ref{long}}.
Define $V(t):=d\phi_{t-t_0}(V)$ where $V$ is a unit tangent vector to $M_{t_0}$.
Note that $V(t)$ is a tangent vector to $M_t$ by construction.
Thus,
$$\arrowvert V\arrowvert '=\frac {g(V',V)}{\arrowvert V\arrowvert}=\frac{\arrowvert V\arrowvert}{\arrowvert \nabla f\arrowvert^2} \Hess (f)(\frac{V}{\arrowvert V\arrowvert},\frac{V}{\arrowvert V\arrowvert})=
\frac{\arrowvert V\arrowvert}{\arrowvert \nabla f\arrowvert^2} \Ric(\frac{V}{\arrowvert V\arrowvert},\frac{V}{\arrowvert V\arrowvert}).$$

Hence,
$$\log\left(\begin{array}{rl}\frac{\arrowvert V\arrowvert(t)}{\arrowvert V\arrowvert(t_0)}\end{array}\right)=
\int_{t_0}^t \frac{\arrowvert V\arrowvert'(s)}{\arrowvert V\arrowvert(s)}ds=
\int_{t_0}^t \frac{1}{\arrowvert \nabla f\arrowvert^2(s)} \Ric\left(\begin{array}{rl}\frac{V(s)}{\arrowvert V\arrowvert(s)},\frac{V(s)}{\arrowvert V\arrowvert(s)}\end{array}\right)ds.$$

The inequality follows from the assumption $\Ric\geq 0$.

\hfill$\square$\\

We deduce the following corollary.

\begin{coro}[Distance comparison] \label{diam}
Let $(M^n,g,\nabla f)$  be a complete steady gradient soliton satisfying the assumptions of Proposition \ref{vol-hyper}. Then,

$$d_{g_{t_0}}\leq d_{\phi^*_{t-t_0}g_t}, \quad \forall t\geq t_0.$$

\end{coro}

\textbf{Remark.}
By the proof of Proposition \ref{long}, we also have the following upper estimate for $t\geq t_0$,
$$d_t\leq e^{\int_{t_0}^t\frac{\sup_{M_s} R}{\arrowvert \nabla f\arrowvert ^2(s)}ds}d_{t_0},$$
which will not be used in this paper.\\

  From now on, we consider the sequence of compact Riemannian manifolds $(M_{t},g_t)_{t\geq t_0}$ for $t_0>0$. In order to take a smooth Cheeger-Gromov limit of this sequence, one has to control the injectivity radius and the curvature and its derivatives of the metrics $g_t$ uniformly.
  
  \begin{lemma}[Injectivity radius of $(M^n,g)$]\label{inj}
  Let $(M^n,g,\nabla f)$ be a complete steady gradient soliton such that
  
  \begin{itemize}
\item[(i)] $\Ric\geq 0$ and $R>0$,
\item[(ii)] $\lim_{+\infty}R=0.$
\end{itemize}

Then for any $t_0>0$, 
$$\inj(M_\leq t)\geq \min\left\{\frac{\pi}{\sqrt{ K_{[t_0,t]}}},\inj(M_{\leq t_0})\right\},\quad \forall t\geq t_0,$$
where $K_{[t_0,t]}$ bounds from above the sectional curvatures of $g$ on $M_{t_0\leq s \leq t}$.\\
In particular, if $(M^n,g)$ has bounded curvature then it has positive injectivity radius.
  \end{lemma}
  
  \textbf{Proof of Lemma \ref{inj}.}
  Let $t>0$ and define the topological retraction $\Pi_t:M^n\rightarrow M^n$ as follows: $\Pi_t(p)=p$ if $f(p)\leq t$ and $\Pi_t(p)=(\phi_{f(p)-t})^{-1}(p)$ otherwise.
  The proof consists in showing that $\Pi_t$ is a distance-nonincreasing map.
  Then one can argue as in the proof of Sharafutdinov [\ref{Sha}] to show the injectivity radius estimate.
  Therefore, we want to show that $\Pi_t$ does not increase distances, i.e.,
  $$d(\Pi_t(p_0),\Pi_t(p_1))\leq d(p_0, p_1), \quad (p_0, p_1 \in M^n).$$
  Let $p_0, p_1\in M^n$, $t_0=f(p_0)$ and $t_1=f(p_1)$.
  Assume w.l.o.g. that $t_0\leq t_1$.
  Consider three cases.
  (1) $t\geq t_1$. There is nothing to prove because $\Pi_t(p_0)=p_0$ and $\Pi_t(p_1)=p_1$.
  (2) $t_0\leq t \leq t_1$. It suffices to show that $s\rightarrow d(p_0, \phi_{s-t}(q_1))$ is a nondecreasing function for $s\geq t$ and $q_1\in M_t$. Take a minimal geodesic $\gamma$ joining $p_0$ to $\phi_{s-t}(q_1)$. Now, $f\circ \gamma$ is a convex function. Thus, $0\leq s-t_0=f(\phi_{s-t}(q_1))-f(p_0)=f(\gamma(1))-f(\gamma(0))\leq g(\nabla f(\gamma(1)), \dot{\gamma}(1)).$ This proves the result.
  (3) $t\leq t_0$. It is equivalent to show that $d(\phi_{t_0-t}(q_0),\phi_{t_1-t}(q_1))\geq d(q_0,q_1),$ for $q_0, q_1\in M_t$. By Proposition \ref{long}, $$d(q_0,q_1)\leq d(\phi_{t_0-t}(q_0),\phi_{t_0-t}(q_1)),$$  
  and by (2),
  $$d(\phi_{t_0-t}(q_0),\phi_{t_0-t}(q_1))\leq d(\phi_{t_1-t_0}(\phi_{t_0-t}(q_1)),\phi_{t_0-t}(q_0)).$$
  This gives the desired inequality.
  
\hfill$\square$\\

\begin{coro}[Diameter estimate]\label{diam-borne}
Let $(M^n,g,\nabla f)$ be a complete steady gradient soliton with bounded curvature satisfying 
\begin{itemize}
\item[(i)] $\Ric\geq 0$, $R>0$ and $\lim_{+\infty} R=0$,
\item[(ii)] $R\in L^1(M^n,g)$.
\end{itemize}
  Then, for any $t_0>0$, there exists a positive constant $D=D(t_0)$ such that,
  $$\diam(g_t)\leq D(t_0),\quad (\forall t\geq t_0).$$
\end{coro}

 \textbf{Proof of Corollary \ref{diam-borne}.}
 On the one hand, by Lemma \ref{inj} and boundedness assumption on curvature, the volume of small balls is uniformly bounded from below. On the other hand, by assumption (ii) and Corollary \ref{vol-borne}, the volume of any tubular neighbourhood of $M_t$ with fixed width is uniformly bounded from above, i.e., for $\alpha>0$, $\vol M_{t-\alpha\leq s \leq t+\alpha}$ is uniformly bounded from above in $t$.
 Therefore, by a ball packing argument,  one can uniformly bound the diameter of $M_t$.
 
  \hfill$\square$\\
We should now estimate the derivatives of the curvature of $g_t$.

\begin{lemma}\label{deriv}
Let $(M^n,g,\nabla f)$ be a complete steady gradient soliton satisfying 
\begin{itemize}
\item[(i)] $\Ric\geq 0$ and $R>0$,
\item[(ii)] $\lim_{+\infty}\arrowvert \Rm(g)\arrowvert=0.$
\end{itemize}
Then there exist constants $(C_k)_{k\geq0}$ depending on $t_0>0$ such that for all $t\geq t_0$, we have
$$\arrowvert \nabla^k\Rm(g_t)\arrowvert\leq C_k.$$
Moreover, $\lim_{t\rightarrow+\infty} \sup_{M_t}\arrowvert \Rm(g_t)\arrowvert=0$.
\end{lemma}

\textbf{Proof of Lemma \ref{deriv}}.\\
The fact that $\lim_{t\rightarrow+\infty} \sup_{M_t}\arrowvert \Rm(g_t)\arrowvert=0$ follows from the Gauss equations:
\begin{equation}
 K_{g_t}(X,Y)=K_g(X,Y)+\det h_t(X,Y), 
 \end{equation}
where $X$ and $Y$ are tangent to $M_t$.

In order to estimate the covariant derivatives $\nabla^{g_t,k} \Rm(g_t)$, it suffices to control those of $\Rm(g)$. 
Indeed, if $A$ is a $p$-tensor and $(X_i)_{0\leq i\leq p}$, $p+1$ tangent vectors to $M_t$, then
$$(\nabla^{g_t}-\nabla^{g})A(X_0,X_1,...,X_p)=\sum_{i=1}^{p}A(X_1,...,(\nabla^{g}_{X_0}X_i-\nabla^{g_t}_{X_0}X_i),...,X_p).$$

Now  $(\nabla^{g}_{X_0}X_i-\nabla^{g_t}_{X_0}X_i)=-h_t(X_0,X_i)\textbf{n}.$
Consequently, 
$$(\nabla^{g_t}-\nabla^g)A=A\ast h_t,$$
where, if $A$ and $B$ are two tensors, $A\ast B$ means any linear combination of contractions of the tensorial product of $A$ and $B$. 
Define $U_k:=(\nabla^{g_t,k}-\nabla^{g,k})A$ for $k\in\mathbb{N^*}$.
Then, 
\begin{eqnarray*}
U_{k+1}&=&(\nabla^{g_t,k+1}-\nabla^{g,k+1})A\\
&=&(\nabla^{g_t}-\nabla^{g})(\nabla^{g_t,k}A)+\nabla^g(\nabla^{g_t,k}-\nabla^{g,k})A\\
&=&(\nabla^{g_t,k}A)\ast h_t +\nabla^gU_k\\
&=&\nabla^{g,k}A\ast h_t+U_k\ast h_t+\nabla^{g}U_k.
 \end{eqnarray*}
 
 By induction on $k$, we show that $U_k$ is a linear combination of contractions of the tensorial products of $(\nabla^{g,i}A)_{0\leq i\leq k-1}$ and $(\nabla^{g,j}h_t)_{0\leq j\leq k-1}$.
 
 Now, bounding  $(\nabla^{g,j}h_t)_{j\geq 0}$ means bounding $(\nabla^{g,i}\Rm)_{i\geq 0}$ and bounding from below $\arrowvert\nabla f\arrowvert$.
If we take $A=\Rm(g)$, we see that bounding $(\nabla^{g_t,k} \Rm(g_t))_{k\geq 0}$ amounts to bounding $(\nabla^{g,k} Rm(g))_{k\geq 0}$ and bounding from below $\arrowvert\nabla f\arrowvert$.
 By Theorem 1.1 in [\ref{Shi}] due to W.X. Shi, there exists
 $T=T(n,\sup_{M^n}\arrowvert K\arrowvert)>0$ and constants $\tilde C_k=\tilde C_k(n,\sup_{M^n}\arrowvert K\arrowvert)$ such that for any time $\tau\in ]0, T]$ and for all $k\geq 0$, one has
$$\sup_{M^n}\arrowvert \nabla^{g(\tau),k}\Rm(g(\tau))\arrowvert^2\leq \frac{\tilde C_k}{\tau^k}.$$
In our situation, the Ricci flow acts by isometries: $g(\tau)=\psi_{\tau}^*g$, for all $\tau \in \mathbb{R}$, where $(\psi_{\tau})_{\tau}$ is the $1$-parameter group of diffeomorphisms of $M^n$ generated by $-\nabla f$.
Modulo a translation at a time slice $0<\tau\leq T$, we can assume 
$$\sup_{M^n}\arrowvert \nabla^{g,k}\Rm(g)\arrowvert^2\leq  C_k,$$
where $C_k=C_k(n,\sup_{M^n}\arrowvert K_g\arrowvert)$.
 This completes the proof.

\hfill$\square$\\

We are now in a position to apply the following theorem to the hypersurfaces $(M_t,g_t)_{t\geq t_{0}}$ assuming that $R\in L^1(M^n,g)$.

\begin{theo}[Cheeger-Gromov]
Let $n\geq 2,(\lambda_i)_{i\geq0}, v, D\in (0,+\infty)$. The class of compact $n$-Riemannian manifolds  $(N^n,g)$ satisfying
$$\arrowvert \nabla^i\Rm\arrowvert\leq \lambda_i,\quad \diam\leq D, \quad v\leq \vol,$$
is compact for the smooth topology.
\end{theo}

Apply this theorem to the sequence  $(M_{k},g_k)_{k\geq t_0}$ with 
$$\lambda_i:=C_i, \quad v=A(t_0),\quad D:=D(t_0),$$
where the sequence $(C_p)_{p}$ comes from Lemma \ref{deriv} and $D(t_0)$ (resp. $A(t_0)$) is obtained by Corollary \ref{diam-borne} (resp. by Corollary \ref{vol-borne}).
There exists a subsequence  $(M_{k_i},g_{k_i})_{i}$ converging to a flat compact manifold  $(M_{\infty},g_{\infty},p_{\infty})$  by assumption on the sectional curvature of $M^n$. As  $M_t$ and $M_{\infty}$ are compact,  the manifolds $M_t$ and $M_{\infty}$ are diffeomorphic for $t>0$.

To sum it up, we have shown the

\begin{prop}\label{type-top}
Let $(M^n,g,\nabla f)$ be a complete steady gradient soliton satisfying 
\begin{itemize}
\item[(i)] $\Ric\geq 0$ and $R>0$,
\item[(ii)] $\lim_{+\infty}\arrowvert\Rm(g)\arrowvert=0,$
\item[(iii)] $R\in L^1(M^n,g)$.
\end{itemize}

Then the level sets of $f$, $M_t$ for $t>0$, are connected and are diffeomorphic to a compact flat $(n-1)$-manifold.
\end{prop}

\textbf{Proof}
The only thing we have to check is the connectedness of the hypersurfaces $M_t$ for $t>0$. 
If $M_t$ have more than one component then $M^n$ would be disconnected at infinity and therefore, by the Cheeger-Gromoll theorem, it would split isometrically as a product $(\mathbb{R}\times N,dt^2+g_0)$  where $N$ is compact. Then, $(N,g_0)$ would be a compact steady gradient soliton, necessarily trivial and so $(M^n,g)$. Contradiction.
This proves the connectedness of the hypersurfaces $M_t$.
\hfill$\square$

\section{The local and global splitting}\label{section2}
As seen in the introduction, a fundamental example of steady soliton discovered by Hamilton is the cigar soliton.
An example of nontrivial steady gradient Ricci solitons with nonnegative sectional curvature and scalar curvature in $L^1$ in higher dimensions is the following: consider the metric product $(\mathbb{R}^2, g_{cigar})\times (\mathbb{R}^{n-2},\eucl)$ and take a  quotient  by a Bieberbach group of rank $n-2$. 

Theorem \ref{global-iso} shows that this is the only example satisfying these assumptions.
We prove this result in section \ref{section2.1} below. A local splitting theorem under weaker curvature assumptions is proved in section \ref{section2.2}.

\subsection{The global splitting}\label{section2.1}

First, we need some background from the theory of nonnegatively curved Riemannian spaces.
The presentation below follows closely Petersen [\ref{Petersen}].
The main difficulty to have a global result comes from the set $M_0=\{f=\min_{M^n} f\}$. Such a set is  \textbf{totally convex}, i.e., any geodesic of $M$ connecting two points of $M_0$ is contained in $M_0$. More generally, any sublevel of a convex function $f$, $M_{\leq t}=\{p\in M^n/f(p)\leq t\}$ is totally convex.
In order to have a better understanding of this notion, we sum up briefly its general properties ([\ref{Petersen}], Chap.11).

\begin{prop}
Let $A\subset (M^n,g)$ be a totally convex subset of a Riemannian manifold.
Then $A$ has an interior which is a totally convex submanifold of $M^n$ and a boundary $\partial A$ non necessarily smooth which satisfies the hyperplane separation property. 
\end{prop}

Moreover, if $A$ is closed, there exists a unique projection on $A$ defined on a neighbourhood of $A$. More precisely, we state Proposition 1.2 of Greene and Shiohama  [\ref{Greene-Sh}] :

\begin{prop}[Greene-Shiohama]
Let $A\subset (M^n,g)$ be a totally convex closed subset . Then there exists an open subset $U\in M^n$ such that 

\begin{itemize}
\item[(i)] $A\subset U$,
\item[(ii)] for any point $p\in U$, there exists a unique point $\pi (p)\in A$ verifying $d(p,A)=d(p,\pi(p))$,
\item[(iii)]the application $\pi:U\rightarrow A$ is continuous,
\item[(iv)] for any $p\in U$, there exists a unique geodesic connecting $p$ to $A$ and it is contained in $U$.
\end{itemize}

If $A$ is compact, we can choose $U$ as $\{p\in M^n/d(p,A)<\epsilon\}$, where $\epsilon$ depends on the compactness of $A$.
\end{prop}

Therefore, the geometric situation near a totally convex closed subset is the same as in the case of  $\mathbb{R}^n$.
In the case of Theorem \ref{global-iso}, we have a smooth convex exhaustion function $f$ and a special  totally convex compact set $\{f=\min_{M^n} f=0\}=M_0$. We would like to understand the topological links (at least) between the levels $M_t$ for $t>0$ and the boundary of the  $\epsilon$-neighbourhood $M_{0,\epsilon}:=\{p\in M^n/d(p,M_0)\leq\epsilon\}$  of $M_0$ for $\epsilon>0$. Cheeger-Gromoll [\ref{Cheeger}] and Greene-Wu [\ref{Greene-Wu}] give a nice answer:

\begin{prop}[Cheeger-Gromoll; Greene-Wu]\label{delta-vois}
Let $f$ be a smooth convex exhaustion function on a Riemannian manifold $(M^n,g)$ with sectional curvature bounded from above.
Then, with the previous notations, for $\epsilon>0$ small enough and $0<\delta<<\epsilon$,
the boundary of the  $\delta$-neighbourhood of $M_0$ and the level $M_{\epsilon}$ are homeomorphic, i.e.
$$\partial M_{0,\delta}\simeq M_{\epsilon}.$$
\end{prop}

Finally, we recall the notion of  \textbf{soul}.
A soul $S\subset M$ of a Riemannian manifold $(M,g)$ is a closed totally convex submanifold.
 This notion has been famous by the Soul theorem [\ref{Cheeger}] by Cheeger and Gromoll.

\begin{theo}[Soul Theorem]
Let $(M^n,g)$ be a complete Riemannian manifold with nonnegative sectional curvature.
Then there exists a soul $S^k$ of $M^n$ such that $M^n$ is diffeomorphic to the normal bundle of $S^k$.
\end{theo}
 
We now are in position to prove Theorem \ref{global-iso} .\\

\textbf{Proof of Theorem \ref{global-iso}}.
First of all, as the Ricci flow acts by isometries in this case, the sectional curvature is nonnegative for any time in $\mathbb{R}$. 
Consider for $p\in M^n$ and for time $\tau\in\mathbb{R}$,
$$\eta(p,\tau):=\{v\in T_pM / \Ric_{g(\tau)}(v)=0\}.$$

We recall that the evolution equation of the Ricci curvature under the Ricci flow $g(\tau)$ satisfies:
$\partial_{\tau}\Ric=\Delta_L\Ric ,$
where $\Delta_L$ means the Lichnerowicz laplacian for the metric $g(\tau)$ acting on symmetric $2$-tensors $T$ by
$\Delta_L T_{ij}:=\Delta T_{ij}+2R_{iklj}T_{kl} -R_{ik}T_{jk}-R_{jk}T_{ik}$.

 Thus, we can use Lemma 8.2 of Hamilton [\ref{Hamilton1}] to claim that $\eta(p,\tau)$ is a smooth distribution invariant by parallel translation and time-independent. Here, time-independence is clear because the flow acts by isometries.
 For any $p\in M^n$, we have an orthogonal decomposition invariant by parallel transport,
 $$T_pM=\eta(p,0)\oplus \{v\in T_pM / \Ric_g(v,v)>0\}=:\eta(p)\oplus\eta^{\perp}(p).$$
 As these distributions are parallel, by the weak de Rham's Theorem [\ref{Petersen}, Chap.8], there exists a neighbourhood $U_p$ for any point $p\in M^n$ such that 
 $$(U_p,g)=(U_1,g_1)\times(U_2,g_2),$$
 where $TU_1=\eta\mid U_1$ and $TU_2=\eta^{\perp}\mid U_2$.\\
 
 \begin{claim}
  $\dim \eta(p)=n-2$ for every $p\in M^n$.\\ 
  \end{claim}
\begin{proof}[Proof of claim 1]

We remind that the second fundamental form $h_t$ of $M_t$ satisfies 
$$h_t=\frac{\Hess (f)}{\arrowvert \nabla f\arrowvert}.$$
Thus, the second fundamental form is nonnegative, i.e., $M_t$ is convex. 
Moreover, the Gauss equation tells us 
\begin{equation}
 K_{g_t}(X,Y)=K_g(X,Y)+\det h_t(X,Y), \label{eq:1}
 \end{equation}
where $X$ and $Y$ are tangent to $M_t$.
Consequently, if we take an orthonormal basis $(E_i)_i$ of $TM_t$, orthogonal to $\textbf{n}$, we have
\begin{equation}
\Ric(g_t)(X,X)=\Ric(g)(X,X)-K_g(X,\textbf{n})+\sum_i \det h_t(X,E_i), \label{eq:2}
\end{equation}
where $X$ is tangent to $M_t$.
Tracing the previous identity, we get
\begin{equation}
R(g_t)=R(g)-2\Ric(g)(\textbf{n},\textbf{n}) +(H_t)^2-\arrowvert h_t\arrowvert^2. \label{eq:3}
\end{equation}

By \eqref{eq:1}, we conclude that we have a family of metrics $g_t$ for $t>0$ on $M_t$ of nonnegative sectional curvature. Now, by Lemma \ref{inj}, we know that $(M^n,g)$ has positive injectivity radius. Therefore, as the scalar curvature is a Lipschitz function since $\nabla R=-2 \Ric(\nabla f)$ is bounded on $M^n$, one has $\lim_{+\infty} \arrowvert \Rm(g)\arrowvert=0$. Consequently, Proposition \ref{type-top} can be applied and $M_t$ is diffeomorphic to a compact flat $(n-1)$-manifold. Therefore, by Bieberbach's theorem ([\ref{Buser}] for a geometric proof), there exists a finite covering $\tilde M_{t}$ of $M_t$ which is topologically a torus $\mathbb{T}^{n-1}$.
To sum it up, we have a family of metrics $\tilde g_t$, for $t>0$, of nonnegative sectional curvature on a $(n-1)$-torus. So, we conclude that the metrics $\tilde g_t$ are flat (and so are the $g_t$) by the equality case in the Bochner theorem [\ref{Petersen}, Chap. 7].
Consequently, the previous identities implies
\begin{eqnarray}
 R&=&2\Ric(\textbf{n},\textbf{n})(>0), \label{eq:4} \\
\Ric(g)(X,X)&=&K_g(X,\textbf{n})\quad \mbox{for any spherical $X$}, \label{eq:5} \\
 \det h_t(X,E_i)&=&0, \quad\mbox{for any i and any spherical $X$}, \label{eq:6} \\
K_g(X,Y)&=&0\quad \mbox{for any spherical plane $(X,Y)$}.
\end{eqnarray}
By \eqref{eq:4},  $\textbf{n}$ is in $\eta^{\perp}$. \eqref{eq:6} means that the rank of $\Ric$ restricted to the hypersurfaces $M_t$ for $t>0$ is at most $1$. Finally, the meancurvature $H_t=R-\Ric(\textbf{n},\textbf{n})=\Ric(\textbf{n},\textbf{n})$ is positive, unless it would contradict \eqref{eq:4}. Thus, the rank of $\Ric$ restricted to the hypersurfaces $M_t$ for $t>0$ is exactly $1$.\\
 \end{proof}

Now, the universal covering $\tilde{M}^n$ of $M^n$ is isometric to $(M_1^2,g_1)\times (M_2^{n-2},g_2)$ where
$TM_1=\{\Ric_g>0\}$ and $TM_2=\{\Ric_g\equiv 0\}$.
Because of the nonnegativity of sectional curvature, $(M_2,g_2)=(\mathbb{R}^{n-2},\eucl)$. Thus $(M_1^2,g_1(\tau))_{\tau\in\mathbb{R}}$ is a complete $2$-dimensional steady gradient soliton with positive scalar curvature. By [\ref{Chow}, corollary B.12], $(M_1^2,g_1(\tau))_{\tau\in\mathbb{R}}$ is necessarily the cigar soliton. 

The last thing we need is the nature of the fundamental group of $M^n$.
\begin{claim}
If $S$ is a soul of $M^n$ then it has codimension $2$ and it is flat. 
\end{claim}

\begin{proof}[Proof of claim 2.]

Indeed, as $S$ is compact and $f$ is convex, $f_{\mid S}$ is constant so $TS\subset \{\Ric_g\equiv 0\}$. Thus, $S$ has codimension at least  $2$ and $S$ is flat since $S$ is totally geodesic in a flat space. Moreover, by Bieberbach's theorem, the rank on $\mathbb{Z}$ of $\pi_1(S)$ is $\dim S$.

 Assume that $S$ has codimension greater than $2$. We obtain a contradiction by linking the fundamental groups $\pi_1(M_t)$ of the hypersurfaces $M_t$ and $\pi_1(S)$ as in Greene-Wu [\ref{Greene-Wu}]. One can assume that $S\subset M_0$ by construction of a soul.
 On the one hand, Lemma \ref{delta-vois} tells us that for $\delta$ small enough, 
 $$\pi_1(M_t)=\pi_1(\partial M_{0,\delta}).$$
 On the other hand, $ M_{0,\delta}$ and the $\delta$-disc bundle $\nu_{\leq\delta}(S):=\{(p,v)\in S\times (T_pS)^{\perp}/\arrowvert v\arrowvert \leq \delta\}$ are homeomorphic by [\ref{Cheeger}].
 Thus, $\pi_1(M_t)=\pi_1(\nu_{\delta}(S))$ where $\nu_{\delta}(S):=\{(p,v)\in S\times (T_pS)^{\perp}/\arrowvert v\arrowvert = \delta\}$.
 Now, the fibre of the fibration $\nu_{\delta}(S)\rightarrow S$ is a $k$-sphere with $k\geq 2$ hence simply-connected since the codimension of $S$ is at least $3$. The homotopy sequence of the fibration shows that $\pi_1(\nu_{\delta}(S))=\pi_1(S)$. To sum it up we have for $t>0$,
 $$\pi_1(S)=\pi_1(M_t).$$
Now the hypersurfaces $(M_t,g_t)$ are flat. In particular, this implies that the rank on $\mathbb{Z}$ of $\pi_1(M_t)$ is $n-1>\dim S=rk_{\mathbb{Z}}(\pi_1(S))$. Contradiction.
\end{proof}

Finally, by [\ref{Cheeger}], we know that the inclusion $S^{n-2}\rightarrow M^n$ is a homotopy equivalence, in particular $\pi_1(M^n)=\pi_1(S)$. So, $\pi_1(M^n)$ is a Bieberbach group of rank $n-2$.

\hfill$\square$\\

\subsection{The local splitting}\label{section2.2}
Without the nonnegativity of sectional curvature, we loose the global splitting.
Nonetheless, under a weaker assumption on the sign of curvature, we still get a local splitting, away from the minimal set $M_0$ of $f$. More precisely,

\begin{theo}\label{local-iso}

Let $(M^n,g,\nabla f)$ be a complete steady gradient soliton such that
\begin{itemize}
\item[(i)] $\Ric\geq 0$ and $R>0$,
\item[(ii)]$\arrowvert\nabla f\arrowvert^2R\geq 2 \Ric(\nabla f,\nabla f)$, 
\item[(iii)]$\lim_{+\infty}\arrowvert \Rm(g)\arrowvert=0,$
\item[(iv)]$R\in L^1(M^n,g)$.
\end{itemize}

Then, $M^n\setminus M_{0}$ is locally isometric to $(\mathbb{R}^2,g_{cigar})\times (\mathbb{R}^{n-2}, \eucl) $.
\end{theo}

\textit{\textbf{Remark.}} Assumption (ii) seems to be ad hoc.
Nonetheless, note that (ii) is verified if the sum of the spherical sectional curvatures is nonnegative, i.e., if for any $t>0$,
$$\sum_{i=1}^{n-1}K_g(X,E_i)\geq 0,$$

where $X$ is tangent to $M_t$ and $(E_i)_{1\leq i\leq n-1}$ is an orthonormal basis of $TM_t$. Note that $R=2\Ric(\textbf{n},\textbf{n})+\sum_{1\leq i,j\leq n-1}K_g(E_i,E_j),$ where $\textbf{n}$ is the unit outward normal to $M_t$.
This condition is clearly implied if $(M^n,g)$ has nonnegative sectional curvature.
Finally, the inequality (ii) is an equality for surfaces. Therefore, Theorem \ref{local-iso} can be seen as a comparison theorem with the geometry of the cigar soliton. \\

\textbf{Proof of Theorem \ref{local-iso}.}

On the one hand, by assumptions (i), (iii) and (iv), we know (Proposition \ref{type-top}) that the hypersurfaces $M_t$ have a finite covering diffeomorphic to a $(n-1)$-torus.
On the other hand, (\ref{eq:3}) in the previous proof,
$$R(g_t)=R(g)-2\Ric(g)(\textbf{n},\textbf{n}) +(H_t)^2-\arrowvert h_t\arrowvert^2,$$
associated with assumption (ii) shows that the hypersurfaces $(M_t,g_t)$ have nonnegative scalar curvature. 
Therefore, we have obtained a sequence of $(n-1)$-torus $(\tilde M_t, \tilde g_t)$ with nonnegative scalar curvature. The Gromov-Lawson theorem (which is relevant in the case $n-1\geq 3$)  [\ref{Lawson}] asserts that the metrics $\tilde g_t$ are flat. Thus, so are the $g_t$ for $t>0$.
Now, along the same lines of the proof of Theorem \ref{global-iso}, we get the following identities  
\begin{eqnarray}
R=2\Ric(\textbf{n},\textbf{n})(>0), \label{1'}\\
\Ric(g)(X,X)=K_g(X,\textbf{n})\quad \mbox{ for any spherical $X$},\label{2'}\\
\det h_t(X,E_i)=0,\quad\mbox{ for any $i$ and any spherical $X$}, \label{3'}\\
K_g(X,Y)=0, \quad\mbox{for any spherical plane $(X,Y)$}.\label{4'}
\end{eqnarray}

In particular, identities (\ref{2'}) and (\ref{4'}) show that the sectional curvature $K_g$ is nonnegative outside the minimal set $M_0$. Here, the flow does not act isometrically on $M^n\setminus M_0$. Nonetheless, for any $p\in M^n\setminus M_0$, there exists a neighbourhood $(U_p, g(t))_{t\in[-T_p,T_p]}$ with $T_p>0$ contained in $(M^n\setminus M_0,g)$ so that the sectional curvature restricted to $(U_p, g(t))_{t\in[-T_p,T_p]}$ remains nonnegative. Thus, as the argument is local, we can use Lemma 8.2 in Hamilton [\ref{Hamilton1}] to claim that $\eta\arrowvert U_p$ is a smooth distribution invariant by parallel translation. According to the weak version of de Rham's theorem, for any $p\in M^n\setminus M_0$, there exists a neighbourhood $U_p$ such that
 $$(U_p,g)=(U_1,g_1)\times(U_2,g_2),$$
 where $TU_1=\eta\mid U_1$ and $TU_2=\eta^{\perp}\mid U_2$.
 We can show, with the same arguments as before, that $\dim \eta(p)=n-2$ for any $p\in M^n\setminus M_0$.\\
 
 \begin{claim}
$\textbf{n}:=\nabla f/\arrowvert \nabla f\arrowvert$ is an eigenfunction for $\Ric$.
\end{claim}
\begin{proof}[Proof of claim 3]
 Let $p\in M_t$ for $t>0$ and $(e_i)_{i=1...n-1}$ an orthonormal basis of $TM_t$ at $p$.
We assume that $\eta(p)$ is generated by $(e_i)_{i=2...n-1}$ and that $\eta^{\perp}(p)$ is generated by $\textbf{n}$ and $e_1$.
By the previous local splitting, 
$$\Ric_g(\textbf{n},e_i)=\Ric_{g_2}(\textbf{n},0)+\Ric_{g_1}(0,e_i)=0,$$
for $i=2...n-1$ and
$$\Ric_g(\textbf{n},e_1)=\Ric_{g_2}(\textbf{n},e_1)=\frac{R_{g_2}(p)}{2}g_2(\textbf{n},e_1)=0,$$
since $\dim \eta^{\perp}(p)=2$!
Consequently, $\Ric$ stabilizes $\textbf{n}$.
\end{proof}

Now, $\Ric$ restricted to $TM_t$ is given by
$$\Ric(X)=\Rm(X,\textbf{n})\textbf{n},$$
for any spherical $X$. As $\Ric$ is a symmetric endomorphism of $TM^n$, it stabilizes $TM_t$ too. Moreover, as $\nabla R+ 2\Ric(\nabla f)=0$, we have for any spherical $X$,
$$g(\nabla R,X)=0.$$
Thus, $R$ and $\arrowvert \nabla f\arrowvert ^2$ are radial functions.

Let $p\in M^n\setminus M_0$ and $U_p$ a neighbourhood such that $(U_p,g)=(U_1^{n-2},g_1)\times(U_2^2,g_2).$ Locally, $g_2$ is
$$g_2=dt^2+\phi^2(t,\theta)d\theta^2,$$ where $\phi$ is a smooth positive function on $U_2$.
We claim that $\phi$ is radial, i.e.  does not depend on $\theta$.
We know that $g=g_1+g_2=dt^2+g_t$ on $U_p$ and $g_t$ is flat, i.e. $\phi^2(t,\theta)d\theta^2+g_1$ is flat. In particular, the coefficients of such a metric are coordinates independent since all the Christoffel symbols vanish. This proves the claim. \hfill$\square$

To sum it up, for any $p \in M^n\setminus M_0$, there exists a neighbourhood $U_p$ such that 
$$(U_p,g)=(U_1^{n-2},g_1)\times(U_2^2,dt^2+\phi^2(t)d\theta^2),$$
where $(U_1,g_1)$ is flat and $R_{g_1}=R_g=-\phi''/\phi>0$.\\
Consequently, $(U_2^2, dt^2+\phi^2(t)d\theta^2)$ is a $2$-dimensional rotationally symmetric steady gradient soliton with positive curvature. An easy calculation ([\ref{Chow}], App. B) shows that 
$\phi(t)=\frac{1}{a}\tanh(at),$ for $a>0$ i.e., $g_2$ is a cigar metric.

\hfill$\square$\\

\subsection{Steady breathers with nonnegative curvature operator and scalar curvature in $L^1$}
We end this section with a rigidity result on steady breathers. Recall that a solution $(M^n,g(t))$  to the Ricci flow is called a steady breather if there exists $T>0$ and a diffeomorphism $\phi$ of $M^n$ such that $g(T)=\phi^*g(0)$. It is quite clear that a steady breather can be extended in an eternal solution. Perelman [\ref{Perelman}] showed that compact steady breathers are compact steady gradient Ricci solitons, hence Ricci-flat. In the noncompact case, the question is still open. In this direction, Hamilton [\ref{Hamilton2}] proved  a more general result on (noncompact) eternal solutions with nonnegative curvature operator.

\begin{theo} [Hamilton]\label{eternal}
If $(M^n,g(t))_{t\in \mathbb{R}}$ is a simply connected complete eternal solution to the Ricci flow with nonnegative curvature operator, positive Ricci curvature and such that  $\sup_{M^n\times \mathbb{R}} R$ is attained at some space and time, then $(M^n,g(t))$ is a steady gradient Ricci soliton.
\end{theo}

Combining this result with Theorem \ref{global-iso}, we get the following corollary.

\begin{coro}\label{Breather}
Let $(M^n, g(t))_{t\in[0,T]}$ be a complete steady breather with nonnegative curvature operator bounded on $M^n\times [0,T]$ and $R_{g(0)}\in L^1(M^n,g(0))$.
Then the universal covering of $M^n$ is isometric to
$$(\mathbb{R}^2,g_{cigar})\times (\mathbb{R}^{n-2}, \eucl), $$
and $\pi_1(M^n)$ is a Bieberbach group of rank $n-2$.
\end{coro}

\textbf{Proof of Corollary \ref{Breather}.}

Recall that ancient solutions with bounded nonnegative curvature operator have nondecreasing scalar curvature [\ref{Chow}, Chap.10], i.e. $R(x,t_1)\leq R(x,t_2)$ for $t_1\leq t_2$ and $x\in M^n$.
Therefore $t\rightarrow \sup_{M^n} R_{g(t)}$ is nondecreasing and is constant on a steady breather.
Now, as $R_{g(0)}\in L^1(M^n,g(0))$, $R_{g(0)}$ is Lipschitz and $K\geq 0$, $\lim_{+\infty} R_{g(0)}=0$ because a Riemannian manifold with bounded nonnegative sectional curvature has positive injectivity radius [\ref{Sha}]. Hence, $\sup_{M^n\times \mathbb{R}} R$ is attained. Consider the universal Riemannian covering $(\tilde{M}^n, \tilde{g}(t))$ of $(M^n, g(t))$. By the Hamilton's maximum principle [\ref{Hamilton1}], $(\tilde{M}^n, \tilde{g}(t))=(N^k,h(t))\times(\mathbb{R}^{n-k}, eucl),$ where $(N^k,h(t))$ is a simply connected complete eternal solution with nonnegative curvature operator, positive Ricci curvature and such that  $\sup_{N^k\times \mathbb{R}} R_{h(t)}$ is attained. By Hamilton's theorem \ref{eternal}, $(N^k,h(t))$ is a steady gradient soliton $(N^k,h,\nabla f)$, and so is $(\tilde{M}^n, \tilde{g})$ with the same potential function $f$. The only thing to check according to Theorem \ref{global-iso} is that $(M^n,g(0))=(M^n,g)$ is a steady gradient soliton, i.e. the potential function $f$ is well-defined on $M^n$.
By [\ref{Cheeger}, section 6], the fundamental group $\pi_1(M^n)$ is a subgroup of $\Isom(N^k,h)\times\Isom(\mathbb{R}^{n-k})$. Let $\psi\in\Isom(N^k,h)$. We want to prove that $\psi^*f=f$. As $\psi$ is an isometry for the metric $h$, $\Hess_h(f-\psi^*f)=0$. Therefore, $\arrowvert\nabla( f-\psi^*f) \arrowvert=Cst=0$ because $N^k$ contains no lines, i.e., $f-\psi^*f=Cst$. Moreover, as $\Ric_h>0$, i.e. $f$ is strictly convex, the scalar curvature attains its maximum at a unique point $p\in N^k$. Thus, $\psi(p)=p$. This proves that $f=\psi^*f$.

\hfill$\square$\\

\section{Scalar curvature decay and volume growth on a steady gradient soliton}\label{section3}

In this section, we try to understand the relations between the scalar curvature decay and the volume growth on a steady gradient soliton.
We recall a result due to Munteanu and Sesum [\ref{Munteanu}].
\begin{lemma}\label{sesum}
Let $(M^n,g,\nabla f)$ be a complete steady gradient soliton.
Then, for any $p\in M^n$, there exists a constant $c_p>0$   such that
$$\vol B(p,r)\geq c_p r,$$
for any $r\geq 1$.
\end{lemma}

What happens if we assume a minimal volume growth on a steady gradient soliton ?
The answer can be given in term of scalar curvature decay:
\begin{lemma}\label{crois-boule-min}
Let $(M^n,g,\nabla f)$ be a complete steady gradient soliton.
 Assume that there exists $C_p>0$ such that
$$\vol B(p,r)\leq C_p r,$$
 for a fixed $p\in M^n$ and $r\geq 1$.\\

Then $R$ belongs to $L^1(M^n,g)$.
\end{lemma}

\textbf{Proof of Lemma \ref{crois-boule-min}.}
Let $p\in M^n$ and $r\geq 1$.
Then, by the Stokes theorem applied to $f$,
$$\int_{B(p,r)}Rd\mu=\int_{B(p,r)}\Delta fd\mu\leq\int_{\partial B(p,r)}\arrowvert \nabla f\arrowvert dA
 \leq CA(p,r),$$
 where $A(p,r)$ is the $(n-1)$-dimensional volume of  the geodesic sphere $S(p,r)=\partial B(p,r)$ and where $C=\sup_{M^n}\arrowvert \nabla f\arrowvert<+\infty$.\\
 Now, $\int_{0}^rA(p,s)ds=VolB(p,r)$. Hence, the volume growth assumption tells us that there exists a sequence of radii $r_k\rightarrow +\infty$ such that the sequence $A(p,r_k)$ is bounded.
 
 Therefore, there exists $C=C(p,\nabla f)$ such that for any $k\in\mathbb{N}$,
 $$\int_{B(p,r_k)}Rd\mu\leq C.$$
 As $M^n=\cup_{k}B(p,r_k)$, $R$ is in $L^1(M^n,g)$.                                  

\hfill$\square$\\
 
We continue with a lemma concerning the "minimal" curvature decay of a steady gradient soliton with nonnegative Ricci-curvature: the scalar curvature decay is at most inversely proportional to the distance in an average sense. More precisely,

\begin{lemma}\label{minimal}
Let $(M^n,g,\nabla f)$ be a complete steady gradient soliton with $\Ric \geq 0$.
Then, for any $p\in M^n$ and every $r>0$,

$$\frac{1}{\vol B(p,r)}\int_{B(p,r)}R d\mu\leq \frac{C}{r},$$

where $C=C(M^n,\nabla f)$.
\end{lemma}

\textbf{Proof of Lemma \ref{minimal}.}
As in the proof of Lemma \ref{crois-boule-min} ,
$$\frac{1}{\vol B(p,r)}\int_{B(p,r)}R d\mu=\frac{1}{\vol B(p,r)}\int_{B(p,r)}\Delta f d\mu\leq
\frac{ \sup_{M^n}\arrowvert \nabla f\arrowvert}{r} \frac{rA(p,r)}{\vol B(p,r)}.$$
Now, by the Bishop-Gromov theorem ([\ref{Zhu}] for a recent and more general proof), 
$$\frac{rA(p,r)}{\vol B(p,r)}\leq n,$$
for any $p\in M^n$ and every $r>0$ since $M^n$ has nonnegative Ricci-curvature.
The result is immediate with $C:=n \sup_{M^n}\arrowvert \nabla f\arrowvert.$
\hfill$\square$\\

We end this section by a remark concerning the vanishing of the geometric invariants  $\lambda_{g,k}(M^n)$ introduced by Perelman [\ref{Perelman}] $(k=1)$ and by Junfang Li [\ref{Junfang}] $(k\geq 1)$.
These invariants are defined for a complete Riemannian manifold $(M^n,g)$ in the following way:
$$\lambda_{g,k}(M^n):=\inf \spec(-4\Delta+kR)=\inf_{\phi\in H_c^{1,2}(M^n)}\frac{\int_{M^n}4\arrowvert\nabla \phi\arrowvert^2 +kR\phi^2d\mu}{\int_{M^n}\phi^2},$$
where the infimum is taken over compactly supported functions in the Sobolev space $H^{1,2}(M^n)$ and where $k>0$.
A sufficient condition to have these invariants well-defined is: $\inf_{M^n} R>-\infty$.
For a complete steady gradient soliton $(M^n,g,\nabla f)$,  $\lambda_{g,k}(M^n)\geq 0$.
Moreover, if a steady soliton is compact, it is Ricci-flat since $\int_{M^n}Rd\mu=\int_{M^n}\Delta fd\mu=0$ by Stokes's theorem. Thus, in this case, $\lambda_{g,k}(M^n)= \inf\spec(-4\Delta)=0$.
What about the noncompact case?
Cheng and Yau [\ref{Cheng}] gave a necessary condition to have 
$\inf\spec(-\Delta)>0$ on a complete manifold.
\begin{prop} \label{crit}
Let $(M^n,g)$ be a complete Riemannian manifold. If the volume growth of geodesic balls is polynomial, i.e. if there exists $C>0$ and $k\geq 0$ such that $\vol B(p,r)\leq Cr^k$ for a fixed $p\in M^n$ and for any $r\geq 1$ then $\inf\spec(-\Delta)=0$.
\end{prop}

Following closely their proof, we obtain the next result for steady gradient solitons.

\begin{prop}\label{lambda}
Let $(M^n,g,\nabla f)$ be a complete steady gradient soliton satisfying $\lambda_{g,k}(M^n)>0$ for some $k>0$.
Then the volume growth of geodesic balls is faster than polynomial, i.e., for any $m\geq 0$ and any $p\in M^n$, there exists $C=C(m,p,k,\nabla f)>0$ such that for $r$ large enough,
$$\vol B(p,r)\geq C r^m.$$

\end{prop}

Before beginning the proof, we state a corollary which follows from Proposition \ref{lambda} and  Bishop-Gromov theorem.

\begin{coro}
Let $(M^n,g,\nabla f)$ be a complete steady gradient soliton with nonnegative Ricci-curvature.
Then, for any $k>0$,

$$\lambda_{g,k}(M^n)=0.$$

\end{coro}

\textbf{Proof of Proposition \ref{lambda}}.
By assumption, we have 
$$\lambda_{g,k}(M^n)\int_{M^n}\phi^2d\mu\leq\int_{M^n}4\arrowvert \nabla \phi\arrowvert^2+kR\phi^2d\mu,$$
for any function with compact support in $H^{1,2}(M^n)$.
Define the following function as in [\ref{Cheng}]:
$$\phi(x)=\left\{\begin{array}{rl}
1 & \mbox{on $B(p,r)$}\\

(2R-r_p(x))/R & \mbox{on $B(p,2R)\setminus B(p,R)$}\\

0 & \mbox{on $M^n\setminus B(p,2R)$}\end{array}\right.$$
for $p\in M^n$ fixed and $R\geq 1$.
The previous inequality applied to this function becomes,
$$\lambda_{g,k}(M^n) \vol B(p,R)\leq 4R^{-2}\vol B(p,2R)+k\int_{M^n}\Delta f \phi^2d\mu.$$
Now,
$$\int_{M^n}\Delta f \phi^2d\mu=-2\int_{M^n}g(\nabla f ,\nabla \phi)\phi d\mu\leq2\sup_{M^n}\arrowvert \nabla f\arrowvert R^{-1}\vol B(p,2R).$$
Thus, 
$$\lambda_{g,k}(M^n) \vol B(p,R)\leq C(M^n,\nabla f)R^{-1}\vol B(p,2R).$$
By lemma \ref{sesum}, there exists $c_p>0$ such that 
$\vol B(p,R)\geq c_pR$, for any $R\geq 1$ .
As $\lambda_{g,k}(M^n)>0$, one has,
$$\vol B(p,R)\geq CR^2,$$
for $R\geq 2$ and $C=C(p,k,\nabla f)$.
The result follows by iterating this argument.

\hfill$\square$\\

{\raggedright Institut Fourier, Université de Grenoble I, UMR 5582 CNRS-UJF,\\
38402, Saint-Martin d'Hères, France.\\
alix.deruelle@ujf-grenoble.fr}


\begin{thebibliography}{99}

\bibitem{B} \label{Buser}Peter Buser, \textit{A geometric proof of Bieberbach's theorems on crystallographic groups.}
Enseign. Math. (2) \textbf{31}, (1985), no. 1-2, 137-145.

\bibitem{C}\label{Carrillo} José A. Carrillo and Lei Ni, \textit{Sharp logarithmic sobolev inequalities on gradient solitons and applications}. arXiv:0806.2417v3.

\bibitem{Che}\label{Cheeger} Cheeger Jeff, Gromoll Detlef, \textit{On the structure of complete manifolds of nonnegative curvature}. Ann. of Math. (2) 96 (1972), 413-443.

\bibitem{Cheng}\label{Cheng} Cheng S.Y. and Yau S.T., \textit{Differential equations on Riemannian manifolds and their geometric applications}, Comm. Pure Appl. Math. 28 (1975), 333-354.

\bibitem{Ch}\label{Chow} Bennett Chow, Peng Lu and Lei Ni, \textit{Hamilton's Ricci Flow}. Lectures in Contemporary Mathematics. American Mathematical society.

\bibitem{Gr1}\label{Greene-Sh} Greene R.E., Shiohama K., \textit{Convex functions on complete noncompact manifolds: topological structure}, Invent. Math. 63 (1981), 129-157.

\bibitem{Gr2}\label{Greene-Wu} Greene R.E., Wu H.,\textit{Nonnegatively Curved Manifolds Which Are Flat Outside a Compact Set}. Proc. Symp. Pure Math. Volume 54, (1993), Part 3.

\bibitem{H1}\label{Hamilton1}Hamilton R..\textit{Four-manifolds with positive curvature operator}. J. Diff. Geom. \textbf{24}, (1986), 153-179.

\bibitem{H2}\label{Hamilton2} Hamilton R. \textit{Eternal solutions to the Ricci flow.} J. Diff. Geom. \textbf{38}, (1993), 1-11.

\bibitem{L}\label{Lawson}H.B. Lawson and M.L. Michelsohn. \textit{Spin Geometry}. Princeton University Press, 1989.

\bibitem{Li}\label{Junfang} Junfang Li. \textit{Eigenvalues and energy functionals with monotonicity formulae under Ricci flow}. arXiv:math/0701548v2.

\bibitem{M}\label{Munteanu} Munteanu O. and Sesum N..\textit{On Gradient Ricci Solitons}. arXiv:0910.1105v1.

\bibitem{Per}\label{Perelman} Perelman Grisha.\textit{The entropy formula for the Ricci flow and its geometric applications}. arXiv:math.DG/0211159.

\bibitem{Pet}\label{Petersen} Peter Petersen. \textit{Riemannian geometry}, volume 171 of \textit{ Graduate Texts in Mathematics}. Springer-Verlag, New York, 1998. 

\bibitem{Sha}\label{Sha} Sharafutdinov V. Pogorelov. \textit{Klingenberg theorem for manifolds homeomorphic to $\mathbb{R}^n$}. Siberian Math. J., \textbf{18} (1977), no. 4, 915-925.

\bibitem{Shi}\label{Shi} W. X. Shi, \textit{Deforming the metric on complete Riemannian manifolds}, J. Differential Geom. \textbf{30} (1989) 223-301.


\bibitem{Z}\label{Zhu} Shunhui Zhu. \textit{The comparison geometry of Ricci curvature}. In Comparison Geometry, eds. Grove and Petersen, MSRI Publ. \textbf{30} (1997), 221-262.

\end{thebibliography}
\end{document}